\documentclass{amsart}
\usepackage{amssymb, amsmath, amsthm, mathtools, hyperref, titlesec, amsaddr}
\makeatletter
\@namedef{subjclassname@2020}{%
	\textup{2020} Mathematics Subject Classification}
\makeatother

\newtheorem{theorem}{Theorem}[section]

\newtheorem{lemma}[theorem]{Lemma}
\newtheorem{corollary}[theorem]{Corollary}
\theoremstyle{definition}

\theoremstyle{remark}
\newtheorem*{remark}{Remark}

\DeclareMathOperator{\rad}{rad} 
 \DeclareMathOperator{\dist}{dist}

\newcommand{\identity}{\mathbf{1}}

\DeclareMathOperator{\Ima}{Im}
\DeclareMathOperator{\Rea}{Re}
\DeclarePairedDelimiter\abs{\lvert}{\rvert}
\newcommand{\field}{\mathbb{C}}
\titleformat{\section}{\normalfont\scshape}{\thesection}{1em}{}
\title{The multiplicative Kowalski-S\l{}odkowski Theorem for Hermitian algebras}

\author{Rudi Brits and Muhammad Hassen}
\address{ University of Johannesburg, Johannesburg, South Africa}
\author{ Cheick Tour\'e}
\address{ Reading, England}
\email{rbrits@uj.ac.za, mhassen@uj.ac.za, cheickkader89@hotmail.com }


\begin{document}
	\begin{abstract}
		We prove, for Hermitian algebras, the multiplicative version of the Kowalski-S\l{}odkowski Theorem which identifies the characters among the collection of all complex valued functions on a Banach algebra $A$ in terms of a spectral condition. Specifically, we show that, if $A$ is a Hermitian algebra, and if $\phi:A\mapsto\mathbb C$ is a continuous function satisfying  $\phi(x)\phi(y) \in \sigma(xy)$ for all $x,y\in A$ (where $\sigma$ denotes the spectrum), then either  $\phi$ or $-\phi$ is a character of $A$; of course the converse holds as well. Our proof depends fundamentally on the existence of positive elements and square roots in these algebras. 
	\end{abstract}
	\maketitle
	
	\section{Introduction}
	
	 Let $A$ be a complex Banach algebra with the identity and zero denoted by $\mathbf 1$ and $\mathbf 0$ respectively. Let $G(A)$ be the invertible group of $A$, and $G_{\mathbf 1}(A)$ be the connected component of $G(A)$ containing $\mathbf 1$.  For $x\in A$, denote by $\sigma(x):=\left\{\lambda\in\mathbb C:\lambda\mathbf 1-x\not\in G(A)\right\}$ the spectrum of $x$.  A nonzero function $\phi:A\to\mathbb C$ (not assumed to be linear) is said to be \emph{multiplicative} if $\phi(xy)=\phi(x)\phi(y)$ for all $x,y\in A$. If $\phi:A\to\mathbb C$ (not assumed to be linear or multiplicative) satisfies
	$\phi(x)\phi(y) \in \sigma(xy)$ for all $x,y\in A,$ then $\phi$ is called \emph{spectrally multiplicative}.  Of course, any nonzero function $\phi:A\to\mathbb C$ which is both linear and multiplicative is called a character of $A$. A classical result, dating back to the late 1960's identifies the characters of $A$ among the linear functionals on $A$: 
	
	\begin{theorem}[Gleason-Kahane-\.{Z}elazko Theorem]\label{origin}
		Let $A$ be a complex unital Banach algebra, and suppose $\phi:A\rightarrow\mathbb C$ is linear. Then $\phi$  is a character of $A$ if and only if $\phi(x)\in\sigma(x)$ for each $x\in A$. 
	\end{theorem}
	
	   For examples showing that Theorem~\refeq{origin} does not extend to real Banach algebras look at \cite[p.43]{Moslehian}. However, in the real case, if one replaces the requirement $\phi(x)\in\sigma(x)$ for each $x\in A$ in the statement of Theorem~\refeq{origin}  by the two conditions $\phi(\mathbf 1)=1$ and $\phi(x)^2+\phi(y)^2\in\sigma(x^2+y^2)$ for all commuting pairs $x,y\in A$, then $\phi$ is a character of $A$; this is due to Kulkarni \cite{Kulkarni}.
	
			A stronger result than Theorem~\ref{origin}, obtained by Kowalski and S\l{}odkowski in \cite{Kowalski}, identifies the characters among all complex-valued functions on $A$ via a spectral condition:

		\begin{theorem}[Kowalski-S\l{}odkowski Theorem]\label{allfunctions}
		Let $A$ be a complex unital Banach algebra. Then a function $\phi:A\rightarrow\mathbb C$ is a character of $A$ if and only if $\phi$  satisfies
		\begin{itemize}
			\item[(i)]{$\phi(\mathbf 0)=0$,}
			\item[(ii)]{$\phi(x)-\phi(y)\in\sigma(x-y)$ for every $x,y\in A$. }
		\end{itemize} 
	\end{theorem}
	
	It is obvious, in Theorem~\ref{allfunctions}, that (i) and (ii) together 
can be replaced by the single requirement $\phi(x)+\phi(y)\in\sigma(x+y)$ for every $x,y\in A.$

     In \cite{Li}, Li, Peralta, Wang and Wang established spherical variants of the Gleason-Kahane-\.{Z}elazko and Kowalski-S\l{}odkowski Theorems which they then use to prove that every weak-2-local
    isometry between two uniform algebras is a linear map. In the following $\mathbb T$ denotes the unit circle in $\mathbb C$. 
	
		\begin{theorem}[Spherical Gleason-Kahane-\.{Z}elazko Theorem; {\cite[Proposition 2.2]{Li}}]\label{sphere1}
		Let $A$ be a complex unital Banach algebra, and suppose $\phi:A\rightarrow\mathbb C$ is a continuous linear functional. If $\phi(x)\in\mathbb T\sigma(x)$ for each $x\in A$, then $\overline{\phi(\mathbf 1)}\phi$ is a character of $A$. 
	\end{theorem}

		\begin{theorem}[Spherical Kowalski-S\l{}odkowski Theorem; {\cite[Proposition 3.2]{Li}}]\label{sphere2}
		Let $A$ be a complex unital Banach algebra, and suppose $\phi:A\rightarrow\mathbb C$ satisfies 
		\begin{itemize}
			\item[(i)]{$\phi(\lambda x)=\lambda\phi(x$) for each $x\in A$, $\lambda\in\mathbb C$,}
			\item[(ii)]{$\phi(x)-\phi(y)\in\mathbb T\sigma(x-y)$ for every $x,y\in A$. }
		\end{itemize}
		Then $\lambda_0\phi$ is a character for some $\lambda_0\in\mathbb T$.  
	\end{theorem}
	
	In \cite{Oi} Oi shows that if one relaxes the homogeneity condition (i) in Theorem~\ref{sphere2}
	to $\phi(\mathbf 0)=0$, then $\phi$ is either complex-linear or conjugate linear and $\overline{\phi(\mathbf 1)}\phi$ is multiplicative.
	
	 The Gleason-Kahane-\.{Z}elazko and Kowalski-S\l{}odkowski Theorems naturally lead to considering the scenario where additivity is sacrificed for multiplicativity. The first major result pertaining to this question was obtained by Maouche: 
	
	\begin{theorem}[Maouche, \cite{Maouche1996}]\label{Maouche}
		Let $A$ be a Banach algebra, and let $\phi:A\rightarrow\mathbb C$ be a multiplicative function satisfying $\phi(x)\in\sigma(x)$ for each $x\in A$. Then, corresponding to $\phi$, there exists a unique character of $A$ which agrees with $\phi$ on $G_{\mathbf 1}(A)$. 
	\end{theorem}
	
	Maouche concludes with an example (where $A=C[0,1]$) showing that $\phi$ in his theorem might not be a character itself. In \cite{Mabrouk} (stated here as Theorem~\ref{MKB}) it is shown that the missing ingredient in his example is continuity of $\phi$. So this prompts the question: If continuity of $\phi$ is added to the hypothesis in Maouche's Theorem, does it then follow that $\phi$ is a character of $A$? If this is indeed the case, then we suspect that the proof might not be easy. Observe also that continuity is not required in the hypotheses of either Theorem~\ref{origin} or Theorem~\ref{allfunctions}, and is part of the conclusion rather than the assumption. Substantial progress on the aforementioned conjecture has been made for algebras with involution. Theorems~\ref{TBS} and \ref{MKB} stated below are the latest results in this context. 
	
	\begin{theorem}[{\cite[Corollary 3.8]{Toure2}}]\label{TBS}
		Let $A$ be a $C^\star$-algebra and let $\phi:A\to\mathbb C$ be a continuous spectrally multiplicative function. Then either $\phi$ or $-\phi$ is a character of $A$. 
	\end{theorem}

	\begin{theorem}[{\cite[Theorem 1.2]{Mabrouk}}]\label{MKB}
		 Let $A$ be a Hermitian algebra and let $\phi:A\to\mathbb C$ be a continuous multiplicative function such that $\phi(x)\in\sigma(x)$ for each $x\in A$. Then $\phi$ is a character of $A$. 
	\end{theorem}

The aim of the current paper is to simultaneously improve on Theorems~\ref{TBS} and \ref{MKB}. Specifically, Theorem~\ref{TBS} holds for all Hermitian algebras; of course this means that the multiplicativity requirement on $\phi$ in Theorem~\ref{MKB} can be relaxed to spectral multiplicativity with the conclusion that either $\phi$ or $-\phi$ is a character of $A$ (in Theorem~\ref{MKB} the fact that $\phi(\mathbf 1)=1$ is immediate from the hypothesis). Expectedly, our proof combines techniques used in 
\cite{Mabrouk} and \cite{Toure2}, but we also need the following result, Theorem~\ref{TSB}, from \cite{Toure1} since there is currently no multiplicative Kowalski-S\l{}odkowski version of Theorem~\ref{Maouche}. Of course, from this emerges another conjecture. If $A$ is a unital Hermitian algebra, we denote by $S_A$ the self-adjoint elements of $A$.

	\begin{theorem}[{\cite[Theorem 3.7, Theorem 3.9]{Toure1}}]\label{TSB}
	Let $A$ be a Hermitian algebra and let $\phi:A\to\mathbb C$ be a continuous spectrally multiplicative function satisfying $\phi(\mathbf 1)=1$. Then the formula $$\psi_\phi(x):=\phi\left(\Rea(x)\right)+i\phi\left(\Ima(x)\right)$$ defines a character of $A$. Further, $\psi_\phi$ agrees with $\phi$ on $G_{\mathbf 1}(A)\cup S_A.$
\end{theorem}

\begin{remark}
	In \cite{Toure1} Theorem~\ref{TSB} was proved for $C^\star$-algebras. A close inspection of the proof shows that the arguments presented there rely only on the fact that self-adjoint elements have real spectra, and the Lie-Trotter formula which says that for $x,y$ in any unital Banach algebra $A$
	$$\lim_{n\rightarrow\infty}\left(e^{x/n}e^{y/n}\right)^n=e^{x+y}.$$	
	Theorem~\ref{TSB} is therefore also valid for Hermitian algebras. 
	In \cite{EPV} Escolano, Peralta, and Villena derive a number of Lie-Trotter formulae for Jordan-Banach algebras and they then use this to derive the multiplicative version of the Gleason-Kahane-\.{Z}elazko Theorem for $JB^\star$-algebras \cite[Theorem 4.3]{EPV}. 
\end{remark}

	\section{Preliminaries} 
	
 Throughout the remainder of this paper $A$ will be a complex unital Hermitian algebra. By definition the \emph{positive} and \emph{strictly positive} elements of $A$ are those self-adjoint elements of $A$ whose spectra are contained in $\mathbb{R}^{\geq 0}$ and $\mathbb{R}^{>0}$, respectively. For $a,b\in S_A$ we will write $a\leq b$ (resp. $a<b$) if $0\leq b-a$ (resp. $0<b-a$). If $\rho$ denotes the spectral radius, then we define the \emph{Pt\'ak functional} by $s(x)\coloneqq \sqrt{\rho(x^\star x)}$, and we recall that $s$ is both subadditive and submultiplicative. Further, since the Pt\'ak functional dominates the spectral radius it is, as in the case with the norm, useful for detecting invertibility. From \cite[Theorem 11.20]{Rudin}, if $0<a$, then $a$ has a positive square root. That is, there exists $y\in A$ such that $y^2=a$, $\sigma(y)\subset(0,\infty)$ and $y=y^\star$. In particular, this square root is not given by a continuous functional calculus but rather by the holomorphic functional calculus \emph{viz}
 \begin{equation}\label{sqrt}
	y=\frac{1}{2\pi i}\int\limits_\Gamma e^{\frac{1}{2}\log\lambda}\left(\lambda\identity-a\right)^{-1}\,d\lambda,
\end{equation}
	 where $\log\lambda $ is the principal branch of the complex logarithm, and $\Gamma$ is a smooth, positively oriented, contour surrounding $\sigma(a)$, not separating $0$ from infinity. The first two properties of $y$ in \eqref{sqrt} are obvious; that $y=y^\star$ is more interesting, and a nice exposition of the argument can be found in \cite{Rudin}. 
	 We denote this positive square root by $y=\sqrt a$.
	 
	   Hermitian algebras are not necessarily semisimple, and in these cases the involution may not be continuous. To deal with this situation (in the forthcoming section) denote by $\rad(A)$ the radical of $A$, and let $\pi:A\to A/\rad(A)$ be the canonical homomorphism. It is elementary to show that $\pi(x)^\star:=\pi(x^\star)$ is a well-defined
	involution on $A/\rad(A)$. Since $\sigma(x)=\sigma(\pi(x))$ for all $x\in A$ it follows that $A/\rad(A)$ is a semisimple Hermitian algebra, and necessarily the involution here is continuous. Moreover, the above spectral equality implies that if $x\in A$ and $x$ is normal modulo $\rad(A)$ i.e $x^\star x-xx^\star\in\rad(A),$ then $s(x)=\rho(x).$

	\section{The multiplicative Kowalski-S\l{}odkowski Theorem for Hermitian algebras}
To simplify our arguments we shall initially assume that $\phi$ satisfies $\phi(\identity)=1$. We start with a few easy but somewhat non-trivial lemmas: 
	 
	\begin{lemma}\label{lem:1}
		Let $\left(a_n\right)$ and $\left(b_n\right)$ be sequences in $S_A$ such that $a_n^2 < b_n^2$ for all $n\in \mathbb{N}$. If $\left(v_n\right)$ is a sequence in $A$ such that $s(b_nv_n)\to 0$, then $s(a_nv_n)\to 0$.
	\end{lemma}
	\begin{proof}
		It suffices to show that $s(a_nv_n)^2\leq s(b_nv_n)^2$ for each $n$. Observe that
		$$s(a_nv_n)^2=\rho((a_nv_n)^\star a_nv_n)=\rho(v_n^\star a_n^2 v_n).$$
		Since $0<b_n^2-a_n^2$ it follows, using the Shirali-Ford Theorem, that $v_n^\star a_n^2 v_n \leq v_n^\star b_n^2 v_n $ for each $n$. Moreover, $v_n^\star b_n^2 v_n \leq \rho(v_n^\star b_n^2 v_n) \identity$ and hence
		$\rho(v_n^\star a_n^2 v_n)\leq \rho(v_n^\star b_n^2 v_n),$ for every $n\in\mathbb{N}.$
		But then, $$s(a_nv_n)^2=\rho(v_n^\star a_n^2 v_n) \leq \rho(v_n^\star b_n^2 v_n) = s(b_nv_n)^2.$$
	\end{proof}

	\begin{lemma}\label{lem:2}
			Let $a\in S_A$, and let $(v_n)$ be a sequence in $A$ such that $s\left(\sqrt{a^2+\frac{1}{n^2}\mathbf 1}\, v_n\right)\to 0$. Then $s(av_n)\to 0$.
	\end{lemma}
	\begin{proof}
		Obviously, from the assumption, $s\left(2\sqrt{a^2+\frac{1}{n^2}\mathbf 1}\, v_n\right)\to 0$. Observe that for each $n\in\mathbb{N}$
		$$0 < \left(\sqrt{a^2+\frac{1}{n^2}\mathbf 1}-a\right)^2 < \left(2\sqrt{a^2+\frac{1}{n^2}\mathbf 1}\right)^2.$$
		Hence, Lemma \ref{lem:1} says that $s\left(\left(\sqrt{a^2+\frac{1}{n^2}\mathbf 1}-a\right)v_n \right)\to 0$. Using the subadditivity of $s$, we get that $s(av_n)\to 0$.
	\end{proof}
	
	Since $xe^x$ is the limit of a sequence of polynomials in $x$ with real coefficients we have the following easy 
	\begin{lemma}\label{lem:0}
		$A$ be a Hermitian algebra and suppose $x\in S_A$. Then $xe^x$  is self-adjoint modulo $\rad(A)$. 
	\end{lemma}
	Lemma~\ref{lem:0} can now be used to obtain the next
	\begin{lemma}\label{lem:3}
		Let $0\leq a$. Then 
		\begin{itemize}
		\item[\rm (a)] $\lim\limits_{n\to\infty}s\left(\sqrt{a+\frac{1}{n^2}\mathbf 1}\,e^{-n\sqrt{a+\frac{1}{n^2}\mathbf 1}}\right)=0,$
		\item[\rm (b)]
		$\lim\limits_{n\to\infty}s \left(\sqrt{a+\frac{1}{n^2}\mathbf 1}\left(\identity+in\sqrt{a+\frac{1}{n^2}\mathbf 1}\right)^{-1}\right)=0.$
		\end{itemize}
	\end{lemma}
	\begin{proof}
		Let $(a_n)$ be a sequence in $A$ with $0<a_n$ for each $n\in \mathbb{N}$, and suppose further that $\{a_n\colon n\in \mathbb{N}\}$ is a commutative subset of $A$. It follows that the bicommutant $B$ of $\{a_n\colon n\in \mathbb{N}\}$ is a Hermitian subalgebra of $A$ preserving the spectra of elements. Consequently, the set $\mathfrak{M}$ of characters of $B$ is nonempty. With $b_n=a_ne^{-na_n}$ we have, using Lemma~\ref{lem:0}, that each $b_n$ is self-adjoint modulo $\rad(B)$ from which it follows that
		$$s(b_n)=\rho(b_n)=\sup\left\{\chi(a_n)e^{-n\chi(a_n)}\colon \chi\in\mathfrak{M}\right\}\leq \sup\left\{te^{-nt}\colon t\geq 0\right\}\leq \frac{1}{en}.$$
		Therefore, $s(b_n)\to 0$. Now let $c_n=a_n(\identity+ina_n)^{-1}$, so that $c_n$ is normal and 
		$$s(c_n)=\rho(c_n)=\sup\left\{\frac{\chi(a_n)}{\abs{1+in\chi(a_n)}}\colon \chi\in\mathfrak{M}\right\}\leq \sup\left\{\frac{t}{\abs{1+int}}\colon t\geq 0\right\}\leq \frac{1}{n}.$$
		Thus, $s(c_n)\to 0$. To get the result, let $a_n=\sqrt{a+\frac{1}{n^2}\mathbf 1}$ for each $n\in\mathbb{N}$.
 	\end{proof}
 	
 	\begin{lemma}\label{lem:4}
 		Let $\phi$ be a continuous spectrally multiplicative function on $A$ with $\phi(\identity)=1$. If $\psi_\phi(x)=0$, then $\phi(x)=0$.
 	\end{lemma}
 	
 	\begin{proof}
 		To simplify notation write $R=\Rea(x)$ and $I=\Ima(x)$. Then $R$ and $I$ are self-adjoint and $\psi_\phi(R)=\psi_\phi (I)=0$. For each $n\in\mathbb{N}$, let $$W_n \coloneqq e^{-n\sqrt{R^2+I^2+\frac{2}{n^2}\mathbf 1}}.$$ Hence, 
 		$$\phi (W_n)=\psi_\phi(W_n)= e^{-\sqrt2}.$$
 		Since $\phi$ is spectrally multiplicative it then follows that 
 		$$\phi(x)e^{-\sqrt2}=\phi(x)\phi(W_n)\in \sigma(xW_n), \text{ for each } n\in\mathbb{N}.$$
 		From Lemma \ref{lem:3}, we have that $s\left(\sqrt{R^2+I^2+\frac{2}{n^2}\mathbf 1}\,W_n\right)\to 0$. Notice that $${R^2+\frac{1}{n^2}\mathbf 1}<{R^2+I^2+\frac{2}{n^2}\mathbf 1},$$ so Lemma \ref{lem:1} says that $s\left(\sqrt{R^2+\frac{1}{n^2}\mathbf 1}\,W_n\right)\to 0$. Now Lemma \ref{lem:2} implies that $s(RW_n)\to 0$. Similarly, $s(IW_n)\to 0$. Therefore,
 		$$s(xW_n) = s(RW_n+iIW_n)\leq s(RW_n)+s(IW_n)\to 0.$$
 		From earlier
 		$${\abs{\phi(x)}}e^{-\sqrt2}\leq \rho(xW_n)\leq s(xW_n),$$
 		which means $\abs{\phi(x)}\leq e^{\sqrt 2}s(xW_n)$. Taking limits as $n\to\infty $ yields $\phi(x)=0$.
 	\end{proof}
 	
 	\begin{lemma}\label{lem:5}
 		Let $\phi$ be a continuous spectrally multiplicative function on $A$ with $\phi(\identity)=1$. If $\alpha\in\field$, and $x\in A$ satisfies $\psi_\phi(x)=0$, then $\phi(\alpha\identity+x)=c_\alpha \alpha$ for some $c_\alpha\in [0,1]$.
 	\end{lemma}
 	\begin{proof}
 		If $\alpha=0$ then the result follows from Lemma~\ref{lem:4}, so take $\alpha\neq 0$. Using $W_n$ as before, let $Y_n=\frac{1}{\alpha} xW_n$ and set $c_\alpha = \phi(\alpha\identity+x)/\alpha$. Then, 
 		$${c_\alpha}{e^{-\sqrt2}}=\frac{1}{\alpha}\phi(\alpha\identity+x)\phi(W_n)\in \sigma(W_n+Y_n).$$
 		Also, $s(Y_n)\to 0$ from the previous proof. Suppose, for the sake of contradiction, that $c_\alpha \not \in [0,1]$ whence $c_\alpha e^{-\sqrt2}\not\in \left[0,e^{-\sqrt2}\right]$. From the definition of $W_n$ one observes that $\sigma(W_n)\subseteq [0,e^{-\sqrt2}]$ for each $n\in\mathbb{N}$. Then we must have that 
 		$$\identity-Y_n\left(c_\alpha e^{-\sqrt2}\identity-W_n\right)^{-1}\not\in G(A),$$
 		for otherwise $c_\alpha e^{-\sqrt2}\identity-W_n-Y_n \in G(A)$ -- a contradiction. But this then means that $$s\left(Y_n\left(c_\alpha e^{-\sqrt2}\identity-W_n \right)^{-1}\right)\geq 1 \mbox{ for all } n\in\mathbb N.$$ Furthermore, $\left(c_\alpha e^{-\sqrt2}\identity-W_n \right)^{-1} $ is normal modulo $\rad(A)$ for each $n$, and hence 
 		$$s\left(\left(c_\alpha e^{-\sqrt2}\identity-W_n \right)^{-1}\right)=\rho\left(\left(c_\alpha e^{-\sqrt2}\identity-W_n\right)^{-1}\right)\leq \frac{1}{\dist\left(\left[0, e^{-\sqrt2}\right], c_\alpha e^{-\sqrt2}\right)}.$$
 		Since $s$ is submultiplicative we get that $s\left(Y_n\left(c_\alpha e^{-\sqrt2}\identity-W_n \right)^{-1}\right)\to 0$, which contradicts the fact that it must be at least $1$ for all $n$. Thus, $\phi(\alpha\identity+x)=c_\alpha\alpha$, with $c_\alpha\in [0,1]$.
  	\end{proof}

  	\begin{lemma}\label{no_zero}
  		Let  $\phi$ be a 
  		continuous   spectrally multiplicative  function $A$ with $\phi(\identity)=1$. If $\alpha \in \mathbb{C},$ and  $x\in A$  satisfies $\psi_\phi(x)=0 $, then $\phi(\alpha\mathbf{1}+ x)\in\{ 0,\alpha  \}.$
  	\end{lemma}
  	\begin{proof}
  		For each $n\in\mathbb N$ let $V_n:=\left(\mathbf{1}+in\sqrt{R^2+I^2+\frac{1}{n^2}\identity} \right)^{-1}.$
  		Again using Lemma~\ref{lem:1}, Lemma~\ref{lem:2} and Lemma~ \ref{lem:3}, we have that $$\lim_ns\left(\sqrt{R^2+I^2+\frac{1}{n^2}\identity}\, V_n\right)={0} \implies \lim_ns\left( R  V_n\right)=0.$$
  		Similarly  $\lim_ns\left( I  V_n\right)=0 $. Observe that each $V_n$ belongs to $ G_{\mathbf 1}(A)$, whence it follows that $\phi(V_n) = \psi_\phi(V_n)=1/2-i/2 $. Let $\alpha \neq 0 $. From Lemma~\ref{lem:5}, we have that $\phi(\alpha\mathbf{1}+x)=c_\alpha \alpha $, with $c_\alpha \in [0,1]$. To obtain the result we have to show that $c_\alpha \in \{0,1\}$: For the sake of a contradiction assume that $ 0< c_\alpha <1$. If we set $Z_n:=\frac{1}{\alpha}xV_n=\frac{1}{\alpha}(R+iI) V_n$, then
  		
  		\begin{equation}\label{3}
  			c_\alpha(1/2-i/2)=\frac{1}{\alpha}\phi(\alpha\mathbf{1}+x)\phi(V_n)\in \sigma\left(V_n+Z_n\right).
  		\end{equation}
  		
  		The first paragraph of the proof shows that $\lim_n s\left(Z_n\right)=0$, and \eqref{3} shows that 
  		$c_\alpha(1/2-i/2)\mathbf{1}-V_n-Z_n  \notin G(A) $. From the definition of $V_n$ together with the spectral mapping theorem, and the fact that $0<\sqrt{R^2+I^2+\frac{1}{n^2}\identity}$ we have that $\sigma\left(V_n\right) \subseteq C_r $, where $C_r$ is the circle in $\mathbb C$ with center $\frac{1}{2}$ and radius $\frac{1}{2}$. Thus $$c_\alpha(1/2-i/2)\mathbf{1} -V_n-Z_n \notin G(A) \mbox{ and } c_\alpha(1/2-i/2)\mathbf{1} -V_n \in G(A) ,$$ which together implies that 
  		\begin{equation}\label{4}
  		 \mathbf{1}-Z_n\left( c_\alpha(1/2-i/2)\mathbf{1} -V_n\right)^{-1} \notin G(A).
  		\end{equation}
  		Since $V_n$ is normal we obtain the estimate
  		\begin{align*}
  		s\left(\left( c_\alpha(1/2-i/2)\mathbf{1} -V_n\right)^{-1}\right)&=\rho\left( \left( c_\alpha(1/2-i/2)\mathbf{1} -V_n\right)^{-1} \right)\\& \leq \frac{1}{\dist(C_r,\{ c_\alpha(1/2-i/2) \} ) }
  		\end{align*}
  		from which it follows that $\lim_ns\left(  Z_n\left( c_\alpha(1/2-i/2)\mathbf{1} -V_n\right)^{-1}\right)=0$, contradicting \eqref{4}. Subsequently $c_\alpha \in \{0,1\}$, and $\phi(\alpha\mathbf{1}+x)\in \{0, \alpha \}$ follows as advertised. 
  	\end{proof}
  	
  	The proof of Theorem~\ref{character} now follows as in \cite{Toure2}:
  	
  	\begin{theorem}\label{character}
  		Let  $\phi$ be a 
  		continuous spectrally multiplicative  function on $A$ with $\phi(\identity)=1$. Then $\phi(x) = \psi_\phi(x)$ for all  $x$ in $A$, and hence $\phi$ is a character of $A$. 
  	\end{theorem}
  	
  	\begin{proof}
  		For $x\in A$ define $K_x:=\{\alpha \in \mathbb{C} : \phi(\alpha\mathbf{1}+x)=0\}$, and assume first that $\psi_\phi(x)=0$. Our aim is to prove that $K_x=\{0\}$. Observe that $0 \in K_x$ (by Lemma~\ref{lem:4}), $K_x \subseteq \sigma(-x)$, and, since $\phi$ is continuous, that $K_x$ is closed. Thus $K_x$ is nonempty and compact.  Let $m$  be a maximum modulus element of  $K_x$. From the definition of $m$ there is a sequence $(k_n) \subset\mathbb C\setminus K_x$ which converges to $m$.  Therefore, by Lemma~\ref{no_zero},  $\lim_n\phi\big(k_n\mathbf{1}+x \big)=\lim_n k_n =m $, and by continuity of $\phi$, $\lim_n\phi\big(k_n\mathbf{1}+x \big)=\phi\big(m\mathbf{1}+x \big)=0 $. Thus $m=0$ from which it follows that $K_x=\{0\}$. Invoking Lemma~\ref{no_zero} again we then obtain $\phi(\alpha\mathbf{1}+x)=\alpha$ for each $\alpha\in\mathbb C$. For any value of $ \psi_\phi(x) $ we use the first part of the proof to deduce that  $$  \phi(x)=\phi\left(  \psi_\phi(x)\mathbf{1}+\left[x-\psi_\phi(x)\mathbf{1}\right] \right)= \psi_\phi(x). $$
  		\end{proof}

  	As a direct consequence of Theorem~\ref{character} we then obtain: 
  	
  	\begin{corollary}\label{character1}
  		Let  $\phi$ be a 
  		continuous spectrally multiplicative  function on $A$. Then  either $\phi$ is a character of $A$ or $-\phi$ is a character of $A$.
  	\end{corollary}
  	\begin{proof}
  		If $\phi$ is spectrally multiplicative then so is $-\phi$. But either $\phi\left(\identity\right)=1$ or $-\phi\left(\identity\right)=1$.  So the result follows from Theorem~\ref{character}.
  	\end{proof}

\end{document}